\documentclass[10pt]{amsart}
\usepackage[margin=1.6in]{geometry}
\usepackage{color}
\newtheorem{theorem}{Theorem}[section]
\newtheorem{lemma}[theorem]{Lemma}
\newtheorem{proposition}{Proposition}[section]

\theoremstyle{definition}
\newtheorem{definition}[theorem]{Definition}

\theoremstyle{remark}
\newtheorem{remark}[theorem]{Remark}

\numberwithin{equation}{section}

\begin{document}

\title[Stationary solutions for the ellipsoidal BGK model in a slab ]{Stationary solutions for the ellipsoidal BGK model in a slab} 

\author{Jeaheang Bang}
\address{Department of Mathematics,
Rutgers University,
Hill Center - Busch Campus
110 Frelinghuysen Road
Piscataway, NJ 08854-8019, USA }
\email{j.bang@rutgers.edu }
\author{Seok-Bae Yun}
\address{Department of Mathematics, Sungkyunkwan University, Suwon 440-746, Republic of Korea}
\email{sbyun01@skku.edu}




\begin{abstract}
We address the boundary value problem for the ellipsoidal BGK model of the Boltzmann equation posed in a bounded interval.
The existence of a unique mild solution is established under the assumption that the inflow boundary data does not concentrate too much around the zero velocity, and
the gas is sufficiently rarefied.
\end{abstract}
\maketitle
\section{Introduction}
In this paper, we are interested in the boundary value problem of stationary ellipsoidal BGK model:
\begin{equation}\label{ESBGK}
  v_{1} \frac{\partial f}{\partial x} =\frac{\rho }{\tau}\big( \mathcal{M}_{\nu}(f)-f \big),
\end{equation}
on a finite interval $[0,1]$ associated with the boundary condition:
\begin{eqnarray}\label{boundary value}
f(0,v)=f_{L}(v)\quad \text{for}\ v_{1}>0,\qquad f(1,v)=f_{R}(v)\quad \text{for}\ v_{1}<0.
\end{eqnarray}
The velocity distribution function $f(x,v)$ is the number density at $x\in [0,1]$ with velocity $v=(v_1,v_2,v_3)\in\mathbb{R}^3$.
We have normalized the spatial domain for simplicity.
$\tau$ is defined by $\tau=\kappa (1-\nu)$, where $\kappa$
denotes the Knudsen number defined as the ratio between the mean free path and the characteristic length,
and $\nu\in (-1/2,1)$ is the relaxation parameter.
The ellipsoidal Gaussian $\mathcal{M}_{\nu}(f)$ reads
\begin{eqnarray*}
\mathcal{M}_{\nu}(f)=\frac{\rho}{\sqrt{\det(2\pi \mathcal{T}_{\nu})}}\exp\left(-\frac{1}{2}(v-U)^{\top}\mathcal{T}^{-1}_{\nu}(v-U)\right),
\end{eqnarray*}
where the local density $\rho$, momentum $U$, temperature $T$ and the stress tensor $\Theta$ are given by
\begin{eqnarray}\label{macroscopic field}
\begin{split}
\rho(x)&=\int_{\mathbb{R}^3}f(x,v)dv,\cr
\rho(x)U(x)&=\int_{\mathbb{R}^3}f(x,v)vdv,\cr
3\rho(x) T(x)&=\int_{\mathbb{R}^3}f(x,v)|v-U|^2dv,\cr
\rho(x)\Theta(x)&=\int_{\mathbb{R}^3}f(x,v)(v-U)\otimes(v-U)dv,
\end{split}
\end{eqnarray}
and the temperature tensor $\mathcal{T}_{\nu}$ is defined as a linear combination of $T$ and $\Theta$:
\begin{eqnarray*}
\mathcal{T}_{\nu}&=&\left(
\begin{array}{ccc}
(1-\nu) T+\nu\Theta_{11}&\nu\Theta_{12}&\nu\Theta_{13}\cr
\nu\Theta_{21}&(1-\nu)T+\nu\Theta_{22}&\nu\Theta_{23}\cr
\nu\Theta_{31}&\nu\Theta_{32}&(1-\nu) T+\nu\Theta_{33}
\end{array}
\right)\cr
&=&(1-\nu) T Id+\nu\Theta.
\end{eqnarray*}
A direct calculation gives the following cancellation property
\begin{eqnarray*}
\begin{split}
\int_{\mathbb{R}^3}\big(\mathcal{M}_{\nu}(f)-f\big)
\left(\begin{array}{c}
1\cr v\cr|v|^2
\end{array}\right)
dv=0,
\end{split}
\end{eqnarray*}
leading to the conservation of following quantities along $x\in[0,1]$:
\[
\int_{\mathbb{R}^3}fv_1dv, \quad\int_{\mathbb{R}^3}fv_1vdv,\quad \int_{\mathbb{R}^3}fv_1|v|^2dv.
\]
Andries et al \cite{ALPP} derived
\begin{eqnarray*}
\int_{\mathbb{R}^3}\big(\mathcal{M}_{\nu}(f)-f\big)\ln fdv\leq 0,
\end{eqnarray*}
which gives the $H$-theorem for the time dependent problem (See also \cite{B3,Yun4}).\newline

The ellipsoidal BGK model is a generalized version of the original BGK model \cite{BGK,Wal} which has been
widely used as a model equation of the Boltzmann equation. It was introduced by Holway \cite{Holway}
to overcome the well-known short-coming of the original BGK model: the incorrect Prandtl number in the Navier-Stokes limit. He introduced the relaxation parameter $\nu$ and generalized the local Maxwellian in the original BGK model into the ellipsoidal Gaussian by replacing the macroscopic temperature with a temperature tensor $\mathcal{T}_{\nu}$ parametrized by $\nu\in(-1/2,1)$ (For the discussion of the range of $\nu$, see \cite{Yun3}).
Through a Chapmann-Enskog expansion, it can be shown that the Prandtl number of the ES-BGK model is given by
$1/(1-\nu)$, and the desired physical Prandtl is obtained by choosing the proper relaxation parameter,
which is $\nu=1-1/Pr\approx-1/2$. Note that, in the case $\nu=0$,
the ES-BGK model reduces back to the original BGK model. Therefore, any results for the ES-BGK model automatically holds for the original BGK model either.
The ES-BGK model, however, has been somewhat neglected in the literature, due mainly to the fact that the $H$-theorem was not verified. This was done recently by Andries et al \cite{ALPP} and revived the interest on this model \cite{ABLP,B1,F-J,G-T,M-W-R,M-S,Stru,Yun2,Yun3,Yun4,Yun5,Z-Stru}. 

In this paper, we consider the ES-BGK model posed in a bounded interval with fixed inflow boundary conditions at both ends. Similar problem was considered by Ukai in \cite{Ukai-BGK} for the original BGK model (the case of $\nu=0$) using a version of the Schauder fixed point theorem to the macroscopic variables. No smallness assumption was imposed, but the  uniqueness was not guaranteed.
We develop here a Banach fixed point type approach for the ES-BGK model which works for the whole range of relaxation parameter ($-1/2<\nu<1$), under the assumptions that the gas is sufficiently rarefied and the boundary inflow data does not concentrate too much near the zero velocity. The first assumption is a kind of smallness condition, which is typical for Banach fixed point type arguments. It is, however, not clear whether the second condition is of intrinsic nature, or mere a technicality that can be overcome by developing finer analysis.
(See Remark (3) in Section 2.)

Brief reference check for related works is in order.
In \cite{ACI}, 1$d$ stationary problem for the Boltzmann equation with Maxwellian molecules was studied in the frame work of measure valued solutions. In a series of paper \cite{AN1,AN2,AN3}, Arkeryd and Nouri studied the existence of weak solutions in $L^1$. Extensions of these argument into the case of two component-gases were made in \cite{B1,B2}. Gomeshi obtained the existence and uniqueness for the Boltzmann equation when the gas is sufficiently rarefied \cite{Gome}, which largely motivated our work. For nice survey of mathematical and physical aspects of the Boltzmann equation and BGK models, see \cite{C,CC,CIP,GL,Sone2,Stru-book,U-T,V}.
%
%
%
%
\subsection{ Notations} To prevent confusion, we fix some notational conventions which we will keep throughout this paper.
\begin{itemize}
\item Every constant, usually denoted by $C$ or $C_{a,b,\cdots}$ will be generically defined.
The  values of $C$ may differ line by line, and even when the same $C$ appears more than once in a line, they are not necessarily of the same value. But all the contants are explicitly computable in principle.
\item We use $C_{\ell,u}$ to denote a positive constant that can be explicitly computed using only
the quantities
in (\ref{quantities}), $\gamma_{\ell}$ in the Theorem \ref{Main} and the relaxation parameter $\nu$.
$C_{\ell,u}$ is generic too in the above mentioned sense.
\item We fix $a_u$, $a_{\ell}$, $a_s$, $c_{u}$, $c_{\ell}$, $c_s$  and $\gamma_{\ell}$ appearing in (\ref{quantities}) and Theorem 2.2 for this use only.
\item When there's no risk of confusion, we write
\[
\mathcal{M}_i(v_i)=e^{-C_{\ell,u}v^2_i}\quad(i=1,2,3)
\]
without explicitly showing the dependence on $C_{\ell,u}$ for simplicity of notations.
\item When there's no risk of confusion, we use $v_1>0$ to denote either $\{v_1>0\}\subset\mathbb{R}$ or $\{v_1>0\}\times\mathbb{R}^2$ according to the context.
\item We define $\sup_x\|\cdot\|_{L^1_2}$ and $\|\cdot\|_{L^{\infty}_2}$ by
\begin{align*}
\sup_x\|f\|_{L^1_2}&=\sup_x\Big\{\int_{\mathbb{R}^3}|f(x,v)|(1+|v|^2)dv\Big\},\cr
\|f\|_{L^{\infty}_2}&=\sup_{x,v}|f(x,v)|(1+|v|^2).
\end{align*}
\end{itemize}

This paper is organized as follows. The main result is stated in the following section 2.
Some relevant issues are also discussed.
In section 3, we reformulate the problem in the
fixed point set up. Some useful technical lemmas are recorded. Then section 4 and 5 is devoted respectively to showing that the solution map is invariant and contractive in the solution space.
\section{Main result}
Before we state our main result, we need to define the following quantities (Recall that $\tau=\kappa(1-\nu)$):
\begin{align}\label{quantities}
\begin{split}
a_{u}&=2\int_{\mathbb{R}^3}f_{LR}dv,\, a_{\ell}=\int_{\mathbb{R}^3}e^{- \frac{ a_u}{\tau|v_1|}}f_{LR}dv,\,
a_s=\int_{\mathbb{R}^3}\frac{1}{|v_1|}f_{LR}dv,\cr
c_{u}&=2\int_{\mathbb{R}^3}f_{LR}|v|^2dv, \, c_{\ell}=\int_{\mathbb{R}^3}
e^{- \frac{ a_u}{\tau|v_1|}}f_{LR}|v|^2dv,\,
c_s=\int_{\mathbb{R}^3}\frac{1}{|v_1|}f_{LR}|v|^2dv,
\end{split}
\end{align}
where we used abbreviated notation:
 \[f_{LR}(v)=f_L(v)1_{v_1>0}+f_R(v)1_{v_1<0}.\]
 We define the mild solution of (\ref{ESBGK}) as follows:
\begin{definition} $f\in L^1_2([0,1]_x\times \mathbb{R}^3_v)$ is said to be a mild solution for (\ref{ESBGK})
if it satisfies
\begin{align}\label{mild f+}
f(x,v)&=e^{-\frac{1}{\tau|v_1|}\int^x_0\rho_f(y)dy}f_{L}(v)\cr
&+\frac{1}{\tau|v_1|}\int_{0}^{x} e^{-\frac{1}{\tau|v_1|}\int^x_y\rho_f(z)dz}\rho_f(y)\mathcal{M}_{\nu}(f)dy
\quad\text{if $v_{1}>0$}
\end{align}
 and
\begin{align}\label{mild f-}
f(x,v)&=e^{-\frac{1}{\tau|v_1|}\int^1_x\rho_f(y)dy}f_{R}(v)\cr
&+\frac{1}{\tau|v_1|}\int_{x}^1 e^{-\frac{1}{\tau|v_1|}\int^y_x\rho_f(z)dz}
\rho_f(y)\mathcal{M}_{\nu}(f)dy
\quad\text{if $v_{1}<0$}.
\end{align}
\end{definition}
%
%
%
%
The main result of this paper is as follows:
\begin{theorem}\label{Main}
Suppose that the mass of the inflow boundary data $f_{LR}\geq 0$ $($not identically zero$)$ is finite and does not concentrate too much around the zero velocity in the following sense:
\begin{eqnarray}\label{satisfies}
f_{LR},\quad \frac{1}{|v_1|}f_{LR}\in L^1_2,
\end{eqnarray}
so that the quantities defined in (\ref{quantities}) are well-defined.  Suppose further that
\begin{equation}\label{0}
\int_{\mathbb{R}^2}f_Lv_idv_2dv_3=\int_{\mathbb{R}^2}f_Rv_idv_2dv_3=0 \quad(i=2,3)
\end{equation}
and there exists a constant $\gamma_{\ell}>0$ such that
\begin{align}\label{gamma}
\left(\int_{v_1>0}e^{-\frac{a_{u}}{\tau|v_1|}}f_L(v)|v_1|dv\right)
\left(\int_{v_1<0}e^{-\frac{a_u}{\tau|v_1|}}f_R(v)|v_1|dv\right)>\gamma_{\ell}.
\end{align}
Then there exists a constant $K>0$ depending only on the quantities defined in (\ref{quantities}) and $\gamma_{\ell}$ such that, if  $\tau>K$,
then there exists a unique mild solution $f\geq0$ for (\ref{ESBGK}) satisfying
\begin{align*}
a_{\ell}\leq \int_{\mathbb{R}^3}f(x,v)dv\leq a_u,\quad c_{\ell}\leq \int_{\mathbb{R}^3}f(x,v)|v|^2dv\leq c_u,
\end{align*}
and \begin{align*}
&\left(\int_{\mathbb{R}^3}fdv\right)\left(\int_{\mathbb{R}^3}f|v|^2dv\right)-\left(\int_{\mathbb{R}^3}fv_1dv\right)^2\geq \gamma_{\ell}.\cr
\end{align*}
\end{theorem}

\begin{remark}
(1) The last assertion of the above theorem guarantees the strict positivity of the temperature tensor
for $k \in\mathbb{R}^3$. This is important since, otherwise, the ellipsoidal Gaussian is not well-defined.
(See Lemma \ref{UT lemma}).\newline\newline
(2) The well-posedness of bulk velocity follows from the above estimates since
\[
\Big| \int_{\mathbb{R}^3}f(x,v)vdv\Big|\leq \frac{a_u+c_u}{2}.
\]

\noindent(3) The condition (\ref{satisfies}) is used to prove the contractiveness of the solution operator.
(See the proof of Proposition \ref{fixed point 2}.)
In the literature on the stationary problem of the Boltzmann equation,
truncation of the collision kernel near origin is often employed  to
overcome the technical difficulties arising in the small velocity region.
(See, for example, \cite{ACI,AN1,CIS}.) Our non-concentration
condition (\ref{satisfies}) can be understood in some sense as a weak truncation of the boundary data near zero. This is good in that we are not imposing any restriction on the equation, but bad
at the same time since it excludes Maxwellian boundary data,
which is the most representative distribution function in the kinetic theory.\newline

\noindent(4) Let us provide an explicit example of the boundary data which satisfies all the conditions above. Define $f_L$ and $f_R$ by
\begin{align*}
f_L(v)=C_L1_{r_1\leq v_1\leq r_2}e^{-\frac{|v_2|^2}{2}}e^{-\frac{|v_3|^2}{2}},\quad f_R(v)=C_R1_{-r_2\leq v_1\leq -r_1}e^{-\frac{|v_2|^2}{2}}e^{-\frac{|v_3|^2}{2}},
\end{align*}
with $C_L,C_R>0$ and $r_2>r_1>0$. Here $1_A$ is the characteristic function on $A$. Clearly, $f_{LR}$ satisfies (\ref{0}). Moreover, since $f_{LR}$ decays sufficiently fast and
vanishes near $v_1=0$, all quantities in (\ref{quantities}) are well-defined. To check (\ref{gamma}), we compute
\[
a_u=\pi(C_L+C_R)(r_2-r_1)
\]
so that
\begin{align*}
\int_{v_1>0}e^{- \frac{ a_u}{\tau|v_1|}}f_{L}(v)v_1dv
&\geq e^{- \frac{\pi(C_L+C_R)(r_2-r_1)}{\tau r_1}}\int_{\mathbb{R}^3}f_{L}(v)v_1dv\cr
&=C_Le^{- \frac{\pi(C_L+C_R)(r_2-r_1)}{\tau r_1}}\Big(\int^{r_2}_{r_1}v_1dv_1\Big)\Big(\int_{\mathbb{R}^2}e^{-\frac{|v_2|^2}{2}}e^{-\frac{|v_3|^2}{2}}dv_2dv_3\Big)\cr
&= \frac{\pi}{4}C_Le^{- \frac{\pi(C_L+C_R)(r_2-r_1)}{\tau r_1}}(r^2_2-r_1^2).
\end{align*}
Similarly,
\begin{align*}
\int_{v_1>0}e^{- \frac{ a_u}{\tau|v_1|}}f_{R}(v)v_1dv\geq \frac{\pi}{4}C_Re^{- \frac{\pi(C_L+C_R)(r_2-r_1)}{\tau r_1}}(r^2_2-r^2_1)
\end{align*}
Therefore,
\begin{align*}
&\left(\int_{v_1>0}e^{-\frac{a_{u}}{\tau|v_1|}}f_L(v)|v_1|dv\right)
\left(\int_{v_1<0}e^{-\frac{a_u}{\tau|v_1|}}f_R(v)|v_1|dv\right)\cr
&\qquad\geq\frac{\pi^2}{4}C_LC_R
e^{- \frac{4(C_R+C_L)(R-r)}{\tau r_1}}(r^2_2-r^2_1)^2,
\end{align*}
which is strictly positive.\newline\newline
\noindent(5)  At first sight, the definition of $a_{\ell}$, $c_{\ell}$ and $\gamma_{\ell}$ seems a little dangerous in that they contain $\tau$ inside the integral, which may
leads to a kind of circular reasoning in the choice of $\tau$.
But, when $\tau$ is very large, (whose size is determined only by $f_{LR}$), we can treat $a_{\ell}$, $c_{\ell}$ and $\gamma_{\ell}$ as if they are independent of $\tau$. We only consider $c_{\ell}$:
Take $r>0$ small enough such that
\begin{align*}
\int_{|v_1|\geq r}f_{LR}(v)|v|^2dv\geq\frac{1}{2}\int_{\mathbb{R}^3}f_{LR}(v)|v|^2dv.
\end{align*}
Then, we observe
\begin{align*}
\int_{\mathbb{R}^3}e^{-\frac{a_{u}}{\tau|v_1|}}f_{LR}(v)|v|^2dv&\geq e^{-\frac{a_{u}}{\tau r}}\int_{|v_1|\geq r}f_{LR}(v)|v|^2dv\cr
&\geq \frac{1}{2}e^{-\frac{a_{u}}{\tau r}}\int_{\mathbb{R}^3}f_{LR}(v)|v|^2dv.
\end{align*}
Since we are going to take $\tau$ sufficiently large, and $a_u$ and $r$ does not depend on $\tau$, we can assume that $\tau$ is large enough such that
$e^{-\frac{a_{u}}{2\tau r_1}}\geq 1/2$
which yields
\begin{align*}
\int_{\mathbb{R}^3}e^{-\frac{a_{u}}{\tau|v_1|}}f_{LR}(v)|v|^2dv\geq\frac{1}{4}\int_{\mathbb{R}^3}f_{LR}(v)|v|^2dv.
\end{align*}
That is,
\[
c_{\ell}\geq \frac{1}{8}c_u.
\]
$a_{\ell}$ and $\gamma_{\ell}$ can be treated similarly.
\end{remark}

%
%
%
%

%
%
%
%
\section{Fixed point Set-up }
We will find the solution for (\ref{ESBGK}) as a fixed point of a solution map defined on the following function space:
\begin{align*}
\Omega=\Big\{f\in L^1_2([0,1]_x\times \mathbb{R}^3_v)~|&~f  \mbox{ satisfies } (\mathcal{A}), (\mathcal{B}), (\mathcal{C})
\Big\}
\end{align*}
endowed with the metric $\displaystyle d(f,g)=\sup_{x\in[0,1]}\|f-g\|_{L^1_2}$, where $(\mathcal{A})$, $(\mathcal{B})$ and $(\mathcal{C})$ denote
\begin{itemize}
\item ($\mathcal{A}$) $f$ is non-negative:
\[
f(x,v)\geq0 \mbox{ for }x,v\in [0,1]\times \mathbb{R}^3.
\]
\item ($\mathcal{B}$) The macroscopic field is well-defined:
\begin{align*}
&a_{\ell}\leq\int_{\mathbb{R}^3}f(x,v)dv\leq a_u,\quad c_{\ell}\leq\int_{\mathbb{R}^3}f(x,v)|v|^2dv\leq c_u.
\end{align*}
\item ($\mathcal{C}$) The following lower bound holds:
\begin{align*}
&\left(\int_{\mathbb{R}^3}fdv\right)\left(\int_{\mathbb{R}^3}f|v|^2dv\right)-\left|\int_{\mathbb{R}^3}fvdv\right|^2\geq \gamma_{\ell}.
\end{align*}
\end{itemize}
In view of (\ref{mild f+}) and (\ref{mild f-}), we define our solution map by
\[
\Phi(f)=\Phi^+(f)1_{v_1>0}+ \Phi^-(f)1_{v_1<0},
\]
where $\Phi^+(f)$ and $\Phi^-(f)$ are
\begin{align}\label{f+}
\begin{split}
\Phi^+(f)(x,v)&=e^{-\frac{1}{\tau|v_1|}\int^x_0\rho_f(y)dy}f_{L}(v)\cr
&+\frac{1}{\tau|v_1|}\int_{0}^{x} e^{-\frac{1}{\tau|v_1|}\int^x_y\rho_f(z)dz}\rho_f(y)\mathcal{M}_{\nu}(f)dy
\quad\text{if $v_{1}>0$}
\end{split}
\end{align}
 and
\begin{align}\label{f-}
\begin{split}
\Phi^-(f)(x,v)&=e^{-\frac{1}{\tau|v_1|}\int^1_x\rho_f(y)dy}f_{R}(v)\cr
&+\frac{1}{\tau|v_1|}\int_{x}^1 e^{-\frac{1}{\tau|v_1|}\int^y_x\rho_f(z)dz}
\rho_f(y)\mathcal{M}_{\nu}(f)dy
\quad\text{if $v_{1}<0$}
\end{split}
\end{align}
Our main Theorem 2.2 then follows directly once we show that the solution map $\Phi$ is invariant and contractive in $\Omega$. The remaining of this paper is devoted to the proof of these two properties, which is stated in Proposition \ref{fixed point 1} and Proposition \ref{fixed point 2} respectively.
Before we move on to the existence proof, we record some technical lemmas that will be useful throughout this paper.
%
%
%
%

\begin{lemma}\emph{\cite{Yun2}}\label{equiv} Suppose $\rho(x)>0$. Then $\mathcal{T}_{\nu}$ satisfies
\begin{eqnarray*}
C^1_{\nu}TId\leq\mathcal{T}_{\nu}\leq C^2_{\nu}TId,
\end{eqnarray*}
where
$C^1_{\nu}=\min\{1-\nu,1+2\nu\}$ and $C^2_{\nu}=\max\{1-\nu,1+2\nu\}$.
\end{lemma}
\begin{lemma}\label{UT lemma}
Let $f\in\Omega$. Then the macroscopic fields $U$, $T$ constructed from $f$  satisfies
\begin{align*}
|U|\leq \frac{a_u+c_u}{2a_{\ell}},\quad\frac{\gamma_{\ell}}{3a_u^2}\leq T\leq \frac{c_u}{3a_{\ell}}
\end{align*}
and
\[
C^1_{\nu}\frac{\gamma_{\ell}}{3a_u^2}|k|^2\leq k^{\top}\mathcal{T}_{\nu}k\leq
C^2_{\nu}\frac{c_u}{3a_{\ell}}|k|^2,
\]
for all $k\in \mathbb{R}^3$.
\end{lemma}
\begin{proof}
For the first inequality, we compute
\begin{align*}
|U|=\frac{|\rho U|}{\rho}=\frac{\Big|\int_{\mathbb{R}^3}fvdv\Big|}{\int_{\mathbb{R}^3}fdv}\leq \frac{a_u+c_u}{2a_{\ell}}.
\end{align*}
Besides, we observe that $T$ is represented as
\begin{align*}
T&=\frac{(3\rho T+\rho|U|^2)-|\rho U|^2\rho^{-1}}{3\rho}\cr
&=\frac{\int_{\mathbb{R}^3}f|v|^2dv-\Big|\int_{\mathbb{R}^3}fvdv\Big|^2\Big(\int_{\mathbb{R}^3}fdv\Big)^{-1}}
{3\int_{\mathbb{R}^3}fdv}.
\end{align*}
We then ignore the last term on the numerator to get
\begin{align*}
T&\leq\frac{\int_{\mathbb{R}^3}f|v|^2dv}
{3\int_{\mathbb{R}^3}fdv}
\leq \frac{c_u}{3a_{\ell}}\cr
\end{align*}
For the lower bound, we recall the definition of $\gamma_{\ell}$ in (\ref{gamma}) to obtain
\begin{align*}
T&=\frac{\Big(\int_{\mathbb{R}^3}fdv\Big)\Big(\int_{\mathbb{R}^3}f|v|^2dv\Big)-\Big|\int_{\mathbb{R}^3}fvdv\Big|^2}
{3\Big(\int_{\mathbb{R}^3}fdv\Big)^2}\geq \frac{\gamma_{\ell}}{3a_{u}^2}.
\end{align*}
Now, the last assertion follows directly from this estimates on $T$ and Lemma \ref{equiv}.
\end{proof}

\begin{lemma}\label{Max decomposition} Let $f\in\Omega$. Then there exists  positive constants $C_{\ell,u}$  depending only on the quantities (\ref{quantities}) and $\gamma_{\ell}$ such that
\[
\mathcal{M}_{\nu}(f)(1+|v|^2)\leq  C_{\ell,u}e^{-C_{\ell,u}|v|^2}.
\]
\end{lemma}
\begin{remark} The two $C_{\ell,u}$ do not necessarily represent the same constant.
\end{remark}
\begin{proof}
We only consider $\mathcal{M}_{\nu}|v|^2$ since the estimate for $\mathcal{M}_{\nu}$ is similar and simpler.
Since $\mathcal{T}_{\nu}$ is symmetric, it is diagonalizable. Let $\lambda_i$ ($i=1,2,3$) denote the eigenvalues. Then,
from Lemma \ref{equiv}, we can deduce that
\[
\det{\mathcal{T}_{\nu}}=\lambda_1\lambda_2\lambda_3\geq \{C^1_{\nu}\}^3T^3.
\]
On the other hand, Lemma \ref{equiv} also implies
\[
-(v-U)^{\top}\mathcal{T}^{-1}_{\nu}(v-U)\leq -\frac{|v-U|^2}{C^2_{\nu}T}.
\]
Therefore, we have
\begin{align*}
\mathcal{M}_{\nu}(f)|v|^2&\leq \frac{\rho}{\{2\pi C^1_{\nu}\}^{3/2}T^{3/2}}e^{-\frac{|v-U|^2}{2C_{\nu}^2T}}|v|^2\cr
&\leq \frac{\rho}{\{2\pi C^1_{\nu}\}^{3/2}T^{3/2}}e^{\frac{|U|^2}{2C_{\nu}^2T}}e^{-\frac{|v|^2}{2C_{\nu}^2T}}|v|^2\cr
&\leq \frac{C C_{\nu}^2\rho }{\{2\pi C^1_{\nu}\}^{3/2}T^{1/2}}e^{\frac{|U|^2}{2C_{\nu}^2T}}e^{-\frac{|v|^2}{4C_{\nu}^2T}}
\end{align*}
where we have used the boundedness: $x^2e^{-x^2}<C$ for some $C>0$.
We then apply the lower and upper bounds of Lemma \ref{UT lemma} to get the desired result.
\end{proof}
%
%
%
%
\section{ $\Phi$ maps $\Omega$ into itself }
The main result of this section is the following proposition, which says that the elements of $\Omega$ is mapped into
$\Omega$ by our solution map $\Phi$:
\begin{proposition}\label{fixed point 1}
Let $f\in\Omega$. Then, under the assumption of Theorem 2.2, we have
\[
\Phi(f)\in \Omega.
\]
\end{proposition}
We divide the proof into Lemma \ref{geq},\ref{al},\ref{au}, and Lemma \ref{gl}.
\begin{lemma}\label{geq} Let $f\in \Omega$. Then
\[
\Phi(f)(x,v)\geq0.
\]
\end{lemma}
\begin{proof}
From the proof of Lemma \ref{Max decomposition}, we see that
\[
\frac{\rho}{\sqrt{\det{2\pi\mathcal{T}_{\nu}}}}\geq\frac{\rho}{\{2\pi C^2_{\nu}T\}^{3/2}},
\]
which, in view of Lemma \ref{UT lemma}, implies
\[
\frac{\rho}{\sqrt{\det{2\pi\mathcal{T}_{\nu}}}}\geq a_{\ell}\left(\frac{3a_{\ell}}{2\pi c_{u}}\right)^{3/2}>0.
\]
Hence, we have
\[
\mathcal{M}_{\nu}(f)>0.
\]
Therefore, we can ignore the second term in (\ref{f+}) to conclude
\begin{align*}
\Phi^+(f)(x,v)\geq e^{-\frac{1}{\tau|v_1|}\int^x_0\rho_f(y)dy}f_{L}(v)\geq0.
\end{align*}
Similarly, we have
\begin{align*}
\Phi^-(f)(x,v)\geq0.
\end{align*}
This completes the proof.
\end{proof}
\begin{lemma}\label{al} Assume $f\in \Omega$. Then we have
\begin{align*}
\int_{\mathbb{R}^3}\Phi(f)dv\geq a_{\ell},\quad \int_{\mathbb{R}^3}\Phi(f)|v|^2dv\geq c_{\ell}.
\end{align*}
\end{lemma}
\begin{proof}
We only prove the second one. Recall from the previous proof that
\begin{eqnarray*}
\Phi(f)\geq e^{-\frac{1}{\tau|v_1|}\int^x_0\rho_f(y)dy}f_{L}(v)1_{v_1>0}
+e^{-\frac{1}{\tau|v_1|}\int^1_x\rho_f(y)dy}f_{R}(v)1_{v_1<0}.
\end{eqnarray*}
Then, since $\rho_f\leq a_u$, we have
\begin{align*}
\Phi(f)&\geq e^{-\frac{xa_u}{\tau|v_1|}}f_{L}(v)1_{v_1>0}+e^{-\frac{(1-x)a_u}{\tau|v_1|}}f_{R}(v)1_{v_1<0}\cr
&\geq e^{-\frac{a_u}{\tau|v_1|}}f_{L}(v)1_{v_1>0}+e^{-\frac{a_u}{\tau|v_1|}}f_{R}(v)1_{v_1<0}\cr
&=e^{-\frac{a_u}{\tau|v_1|}}f_{LR}.
\end{align*}
Integrating with respect to $|v|^2dv$, we obtain the desired lower bound:
\begin{align*}
\int_{\mathbb{R}^3}\Phi(f)|v|^2dv
\geq\int_{\mathbb{R}^3}e^{-\frac{a_u }{\tau|v_1|}}f_{LR}|v|^2dv
\geq c_{\ell}.
\end{align*}

\end{proof}
\begin{lemma}\label{au} Let $f\in \Omega$. Then we have
\begin{align}\label{numberr}
\int_{\mathbb{R}^3}\Phi(f)dv\leq a_u,
\quad
\int_{\mathbb{R}^3}\Phi(f)|v|^2dv\leq c_u.
\end{align}
\end{lemma}
\begin{proof}
We only prove the second one. Consider
\begin{align*}
\int_{\mathbb{R}^3}\Phi^+(f)|v|^2dv&=\int_{v_1>0}e^{-\frac{1}{\tau|v_1|}\int^x_0\rho_f(y)dy}f_{L}(v)|v|^2dv\cr
&+\int_{v_1>0}\int_{0}^{x}\frac{1}{\tau|v_1|}e^{-\frac{\int^x_y\rho_f(z)dz}{\tau|v_1|}}\rho_f(y)\mathcal{M}_{\nu}(f)|v|^2dydv.
\end{align*}
Using $\rho_f\geq a_{\ell} $, we estimate the first term of $\int_{\mathbb{R}^3}\Phi^+|v|^2dv$ as
\begin{eqnarray*}
\int_{v_1>0}e^{-\frac{1}{\tau|v_1|}\int^x_0\rho_f(y)dy}f_{L}(v)|v|^2dv
\leq \int_{v_1>0}e^{-\frac{a_{\ell}x}{\tau|v_1|}}f_{L}(v)|v|^2dv
\leq\int_{v_1>0}f_{L}(v)|v|^2dv.
\end{eqnarray*}
By Lemma \ref{Max decomposition}, we compute
\begin{align*}
&\int_{v_1>0}\int_{0}^{x}\frac{1}{\tau|v_1|}e^{-\frac{\int^x_y\rho_f(z)dz}{\tau|v_1|}}\rho_f(y)\mathcal{M}_{\nu}(f)|v|^2dydv\cr
&\quad\leq C_{\ell,u}\left\{\int_{0}^{x}\Big(\int_{v_1>0} \frac{1}{\tau|v_1|}e^{-\frac{\int^x_y\rho_f(z)dz}{\tau|v_1|}}
\mathcal{M}_1(v_1)dv_1\Big)\rho_f(y)dy\right\}
\left\{\int_{\mathbb{R}^2}\mathcal{M}_{2}\mathcal{M}_{3}dv_2dv_3\right\}\cr
&\quad\leq C_{\ell,u}a_u\int^x_0\int_{v_1>0} \frac{1}{\tau|v_1|}e^{-\frac{\int^x_ya_{\ell}dz}{\tau|v_1|}}\mathcal{M}_1(v_1)dv_1dy\cr
&\quad\leq C_{\ell,u} a_u\int_{0}^{x}\int_{v_1>0} \frac{1}{\tau|v_1|}e^{-\frac{a_{\ell}(x-y)}{\tau|v_1|}}\mathcal{M}_1(v_1)dv_1dy\cr
&\quad\equiv C_{\ell,u}a_uI.
\end{align*}
Then divide the domain of integration into the following two regions
\begin{eqnarray*}
I&=&\left\{\int^x_0\int_{|v_1|<\tau}+\int^x_0\int_{|v_1|>\tau}\right\}\frac{1}{\tau|v_1|}e^{-\frac{a_{\ell}(x-y)}{\tau|v_1|}}\mathcal{M}_1(v_1)dv_1dy\cr
&\equiv&I_1+I_2.
\end{eqnarray*}
For $I_1$, we carry out the integration on $x$ first:
\begin{align*}
I_1&=\int_{|v_1|<\tau}\left\{\int_{0}^x\frac{1}{\tau|v_1|}e^{-\frac{a_{\ell}(x-y)}{\tau|v_1|}}dy\right\}\mathcal{M}_1(v_1)dv_1\cr
&=\frac{1}{a_{\ell}}\int_{|v_1|<\tau}\left\{1-e^{-\frac{a_{\ell}x}{\tau|v_1|}}\right\}\mathcal{M}_1(v_1)dv_1\cr
&\leq\frac{1}{a_{\ell}}\int_{|v_1|<\tau}\left\{1-e^{-\frac{a_{\ell}}{\tau|v_1|}}\right\}\mathcal{M}_1(v_1)dv_1,
\end{align*}
and decompose the integration further as
\begin{eqnarray*}
\frac{1}{a_{\ell}}\left\{\int_{|v_1|<\frac{1}{\tau}}+\int_{\frac{1}{\tau}<|v_1|<\tau}\right\}\left\{1-e^{-\frac{a_{\ell}}{\tau|v_1|}}\right\}\mathcal{M}_1(v_1)dv_1\equiv I_{11}+I_{12}.
\end{eqnarray*}
We use rough estimates $1-e^{-\frac{a_{\ell}}{\tau|v_1|}}\leq1$ and $\mathcal{M}_1(v_1)\leq 1$ to control $I_{11}$ by
\begin{eqnarray*}
\frac{1}{a_{\ell}}\int_{|v_1|<\frac{1}{\tau}}dv_1
\leq\frac{1}{a_{\ell}}\frac{1}{\tau}.
\end{eqnarray*}
On the other hand, we expand $1-e^{-\frac{a_{\ell}}{\tau|v_1|}}$ by Taylor series and use $\mathcal{M}_1(v_1)\leq 1$ to find
\begin{align*}
I_{12}&=\frac{1}{a_{\ell}}\int_{\frac{1}{\tau}<|v_1|<\tau}
\left\{\left(\frac{a_{\ell}}{\tau|v_1|}\right)-\frac{1}{2!}\left(\frac{a_{\ell}}{\tau|v_1|}\right)^2+\frac{1}{3!}\left(\frac{a_{\ell}}{\tau|v_1|}\right)^3+\cdots\right\}\mathcal{M}_1(v_1)dv_1\cr
&\leq\frac{1}{a_{\ell}}\Big|\int_{\frac{1}{\tau}}^{\tau}\frac{a_{\ell}}{\tau r}\,dr\Big|+\Big|\int_{\frac{1}{\tau}}^{\tau}\frac{1}{2!}\left(\frac{a_{\ell}}{\tau r}\right)^2dr\Big|+
\Big|\int_{\frac{1}{\tau}}^{\tau}\frac{1}{3!}\left(\frac{a_{\ell}}{\tau r}\right)^3dr \Big|+\cdots\cr
&=\left\{\frac{1}{\tau}\ln \tau^2+\frac{1}{2!}\frac{a_{\ell}}{\tau^2}\frac{\tau^2-1}{\tau}
+\frac{1}{2\cdot 3!}\frac{a^{2}_{\ell}}{\tau^3}\frac{\tau^4-1}{\tau^2}
+\frac{1}{3\cdot 4!}\frac{a^{3}_{\ell}}{\tau^4}\frac{\tau^6-1}{\tau^3}
\cdots\right\}\cr
&\leq \frac{1}{\tau}\ln \tau^2+\frac{e^{a_{\ell}}}{a_{\ell}}\frac{1}{\tau}.
\end{align*}
In the last line, we used
\begin{align*}
&\frac{1}{2!}\frac{a_{\ell}}{\tau^2}\frac{\tau^2-1}{\tau}
+\frac{1}{2\cdot 3!}\frac{a^{2}_{\ell}}{\tau^3}\frac{\tau^4-1}{\tau^2}+
\frac{1}{3\cdot 4!}\frac{a^{3}_{\ell}}{\tau^4}\frac{\tau^6-1}{\tau^3}
\cdots\cr
&\qquad=\frac{1}{2!}\frac{a_{\ell}}{\tau}\frac{\tau^2-1}{\tau^2}
+\frac{1}{2\cdot 3!}\frac{a^{2}_{\ell}}{\tau}\frac{\tau^4-1}{\tau^4}+
\frac{1}{3\cdot 4!}\frac{a^{3}_{\ell}}{\tau}\frac{\tau^6-1}{\tau^6}
\cdots\cr
&\qquad \leq\frac{1}{a_{\ell}}\left\{\frac{a^2_{\ell}}{2!}+\frac{a^3_{\ell}}{3!}+\frac{a^4_{\ell}}{4!}\cdots\right\}\frac{1}{\tau}\cr
&\qquad\leq\frac{e^{a_{\ell}}}{a_{\ell}}\frac{1}{\tau},
&
\end{align*}
which again is the consequence of $(\tau^n-1)/\tau^n<1$.
Finally, $I_2$ is estimated as follows:
\begin{eqnarray*}
I_2
&\leq&\int_{|v_1|>\tau}\left\{\int_{0}^1\frac{1}{\tau|v_1|}e^{-\frac{a_{\ell}(x-y)}{\tau|v_1|}}dy\right\}\mathcal{M}_1(v_1)dv_1\cr
&\leq&\int_{|v_1|>\tau}\left\{\int_{0}^1\frac{1}{\tau|v_1|}dy\right\}\mathcal{M}_1(v_1)dv_1\cr
&\leq&\frac{1}{\tau^2}\int_{|v_1|>\tau}\mathcal{M}_1(v_1)dv_1\cr
&\leq&\frac{1}{\tau^2}\int_{\mathbb{R}^3}\mathcal{M}_1(v_1)dv_1\cr
&\leq& C_{\ell,u}\frac{1}{\tau^2}.
\end{eqnarray*}
The above decomposition of velocity domain is largely motivated from \cite{Gome}. In conclusion, we obtain the following estimate for $I$:
\begin{equation}\label{I}
I\leq C_{\ell,u}\left\{\frac{1}{\tau}\ln \tau^2+\frac{1}{\tau}
+\frac{1}{\tau^2}\right\}\leq C_{\ell,u}\left(\frac{\ln \tau+1}{\tau}\right),
\end{equation}
where $C_{\ell,u}>0$ depends only on $\nu$, quantities in (\ref{quantities}) and $\gamma_{\ell}$.
Finally, we gather these estimates to obtain
\begin{eqnarray*}
\int_{v_1>0}\Phi^+(f)|v|^2dv\leq \int_{v_1>0}f_L(v)|v|^2dv+C_{\ell,u}\left(\frac{\ln \tau+1}{\tau}\right).
\end{eqnarray*}
By an identical argument, similar estimate can be derived for $\Phi^-(f)$:
\begin{eqnarray*}
\int_{v_1<0}\Phi^-(f)dv\leq \int_{v_1<0}f_R(v)|v|^2dv+C_{\ell,u}\left(\frac{\ln \tau+1}{\tau}\right).
\end{eqnarray*}
Summing up these estimates and recalling the definition of $c_u$, we get
\begin{eqnarray*}
\int_{\mathbb{R}^3}\Phi(f)|v|^2dv\leq \frac{1}{2}c_{u}+C_{\ell,u}\left(\frac{\ln \tau+1}{\tau}\right),
\end{eqnarray*}
which gives the second estimate in (\ref{numberr}) for sufficiently large $\tau$.
\end{proof}
\begin{lemma}\label{23}
For $i=2,3$, we have
\[
\Big|\int_{\mathbb{R}^3}\Phi(f)v_idv\Big|\leq C_{\ell,u}\left(\frac{\ln \tau+1}{\tau}\right).
\]
\end{lemma}
\begin{proof}
We only prove for $i=2$.
We integrate (\ref{f+}) with respect to $v_2dv_2dv_3$ to get
\begin{eqnarray*}
F_+(x,v_1)=e^{-\frac{1}{\tau |v_1|}\int^x_0\rho(y)dy}F_{L,+}(v_1)
+\frac{1}{\tau|v_1|}\int^x_0e^{-\frac{1}{\tau |v_1|}\int^x_y\rho(z)dz}\rho(y)G(y,v_1)dy.
\end{eqnarray*}
where
\begin{align*}
F_+=\int_{\mathbb{R}^2}\Phi_+(f)v_2dv_2dv_3,~ G=\int_{\mathbb{R}^2}\mathcal{M}_{\nu}(f)v_2dv_2dv_3,
~ F_{L,+}(v_1)=\int_{\mathbb{R}^2}f_L(v)v_2dv_2dv_3.
\end{align*}
By our assumption on $f_L$, we have
\begin{align*}
F_{L,+}(v_1)=0,
\end{align*}
so that
\begin{eqnarray}\label{sothat}
F_+(x,v_1)=\frac{1}{\tau|v_1|}\int^x_0e^{-\frac{1}{\tau |v_1|}\int^x_y\rho(z)dz}\rho(y)G(y,v_1)dy.
\end{eqnarray}
Using Lemma \ref{Max decomposition}, we compute
\begin{align*}
G(x,v_1)
\leq C_{\ell,u} \mathcal{M}_1\Big(\int_{\mathbb{R}^2}\mathcal{M}_2\mathcal{M}_3|v_2|dv_2dv_3\Big)
\leq C_{\ell,u}\mathcal{M}_1(v_1).
\end{align*}
Substituting this into (\ref{sothat}),
\begin{align*}
F_+(x,v_1)&\leq C_{\ell,u}\frac{1}{\tau|v_1|}\int^x_0e^{-\frac{1}{\tau |v_1|}\int^x_y\rho(z)dz}\rho(y)\mathcal{M}_1(v_1)dy\cr
&\leq C_{\ell,u}\frac{a_u}{\tau|v_1|}\int^x_0e^{-\frac{a_{\ell}(x-y)}{\tau |v_1|}}\mathcal{M}_1(v_1)dy\cr
\end{align*}
Now, integrating on $v_1>0$ and recalling (\ref{I}), we get
\begin{align*}
\int_{v_1>0}F_+(x,v_1)dv_1
&=C_{\ell,u}\int_{0}^{x}\int_{v_1>0} \frac{1}{\tau|v_1|}e^{-\frac{a_{\ell}(x-y)}{\tau|v_1|}}\mathcal{M}_1(v_1)dv_1dy\cr
&\leq C_{\ell,u}\left(\frac{\ln \tau+1}{\tau}\right).
\end{align*}
That is,
\begin{align*}
\int_{v_1>0}\Phi_+(f)v_2dv
\leq C_{\ell,u}\left(\frac{\ln \tau+1}{\tau}\right).
\end{align*}
Applying the same type of argument to $\Phi_-(f)$, we can derive
\begin{align*}
\Big|\int_{v_1<0}\Phi_-(f)v_2dv\Big|
\leq C_{\ell,u}\left(\frac{\ln \tau+1}{\tau}\right).
\end{align*}
We sum these up to obtain the desired result.
\end{proof}
\begin{lemma}\label{gl} Let $f\in \Omega$. Then, for sufficiently large $\tau$, we have
\begin{align*}
&\left(\int_{\mathbb{R}^3}\Phi(f)dv\right)\left(\int_{\mathbb{R}^3}\Phi(f)|v|^2dv\right)-\left|\int_{\mathbb{R}^3}\Phi(f)vdv\right|^2\geq \gamma_{\ell}.
\end{align*}
\end{lemma}
\begin{proof}
Since we have shown $\Phi(f)\geq0$ in Lemma \ref{geq}, we can apply the Cauchy-Schwarz inequality as
\begin{align*}
&\left(\int_{\mathbb{R}^3}\Phi(f)dv\right)\left(\int_{\mathbb{R}^3}\Phi(f)|v|^2dv\right)
-\left|\int_{\mathbb{R}^3}\Phi(f)vdv\right|^2\cr
&\hspace{1cm}\geq\left(\int_{\mathbb{R}^3}\Phi(f)|v|dv\right)^2-\left|\int_{\mathbb{R}^3}\Phi(f)vdv\right|^2\cr
&\hspace{1cm}\geq\left(\int_{\mathbb{R}^3}\Phi(f)|v_1|dv\right)^2-\left|\int_{\mathbb{R}^3}\Phi(f)vdv\right|^2.
\end{align*}
In the last line, we have used $|v|\geq |v_1|$. Then we decompose the last term as
\begin{align*}
&\left(\int_{\mathbb{R}^3}\Phi(f)|v_1|dv\right)^2-\Big|\int_{\mathbb{R}^3}\Phi(f)vdv\Big|^2\cr
&\qquad \geq\left(\int_{\mathbb{R}^3}\Phi(f)|v_1|dv\right)^2-\Big\{\sum_{1\leq i\leq 3}\Big|\int_{\mathbb{R}^3}\Phi(f)v_idv\Big|\Big\}^2\cr
&\qquad=\left(\int_{\mathbb{R}^3}\Phi(f)|v_1|dv\right)^2
-\Big(\int_{\mathbb{R}^3}\Phi(f)v_1dv\Big)^2-R\cr
&\qquad\equiv I-R.
\end{align*}
where
\begin{align*}
R=M^2_2+M^2_3+2M_1M_2+2M_2M_3+2M_3M_1,
\end{align*}
for  $M_i=\Big|\int_{\mathbb{R}^3}\Phi(f)v_idv\Big|$. Since $M_1$ is bounded: $M_1\leq a_u+c_u$, we see from Lemma \ref{23} that $R$ can be made arbitrarily small by
taking $\tau$ sufficiently large:
\[
R\leq C_{\ell,u}\left(\frac{\ln \tau+1}{\tau}\right).
\]
For the estimate of $I$, we use the simple identity: $a^2-b^2=(a-b)(a+b)$ to bound $I$ from below by
\begin{align*}
I&\geq\left\{\int_{\mathbb{R}^3}\Phi(f)(|v_1|+v_1)dv\right\}\left\{\int_{\mathbb{R}^3}\Phi(f)(|v_1|-v_1)dv\right\}\cr
&=4\left\{\int_{v_1>0}\Phi(f)|v_1|dv\right\}\left\{\int_{v_1<0}\Phi(f)|v_1|dv\right\}.
\end{align*}
We then recall from (\ref{f+}) that
\[
\Phi(f)\geq e^{-\frac{1}{\tau|v_1|}\int^x_0\rho_f(y)dy}f_L1_{v_1>0}+e^{-\frac{1}{\tau|v_1|}\int^1_x\rho_f(y)dy}f_R1_{v_1<0}
\]
to obtain
\begin{align*}
&4\left\{\int_{v_1>0}\Phi(f)|v_1|dv\right\}\left\{\int_{v_1<0}\Phi(f)|v_1|dv\right\}\cr
&\hspace{1cm}\geq4\left(\int_{v_1>0}e^{-\frac{1}{\tau|v_1|}\int^x_0\rho_f(y)dy}f_L|v_1|dv\right)
\left(\int_{v_1<0}e^{-\frac{1}{\tau|v_1|}\int^1_x\rho_f(y)dy}f_R|v_1|dv\right)\cr
&\hspace{1cm}\geq4\left(\int_{v_1>0}e^{-\frac{a_{u}}{\tau|v_1|}}f_L|v_1|dv\right)
\left(\int_{v_1<0}e^{-\frac{a_u}{\tau|v_1|}}f_R|v_1|dv\right).
\end{align*}
In view of (\ref{gamma}), we see that the last term is bounded from below by $4\gamma_{\ell}$.
In summary, we have derived the following estimate:
\begin{align*}
&\left(\int_{\mathbb{R}^3}\Phi(f)dv\right)\left(\int_{\mathbb{R}^3}\Phi(f)|v|^2dv\right)
-\left|\int_{\mathbb{R}^3}\Phi(f)vdv\right|^2\cr
&\hspace{1.5cm} \geq4\gamma_{\ell}-C_{\ell,u}\left(\frac{\ln \tau+1}{\tau}\right).
\end{align*}
Therefore, upon choosing sufficiently large $\tau$, we can get  the desired result.
\end{proof}
\section{$\Phi$ is contractive in $\Omega$}
The goal of this section is to show that the solution map $\Phi$ is contractive in $\Omega$. First, we consider the continuity property of the ellipsoidal Gaussian.
\begin{proposition}\label{Lipshitz} Let $f$, $g$ be elements of $\Omega$. Then the non-isotropic Gaussian $\mathcal{M}_{\nu}$ satisfies the following continuity property:
\begin{eqnarray*}
|\mathcal{M}_{\nu}(f)-\mathcal{M}_{\nu}(g)|\leq C_{\ell,u}\sup_x\|f-g\|_{L^{1}_2}
e^{-C_{\ell,u}|v|^2}.
\end{eqnarray*}
\end{proposition}
\begin{proof}
(1) We expand $\mathcal{M}_{\nu}(f)-\mathcal{M}_{\nu}(g)$ as
\begin{eqnarray}\label{turnbackto}
\begin{split}
\mathcal{M}_{\nu}(f)-\mathcal{M}_{\nu}(g)
&=(\rho_{f}-\rho_{g})
\int^1_0\frac{\partial\mathcal{M}_{\nu}(\theta)}{\partial \rho}d\theta\cr
&+(U_{f}-U_{g})\int^1_0\frac{\partial\mathcal{M}_{\nu}(\theta)}{\partial U}d\theta\cr
&+(\mathcal{T}_{f}-\mathcal{T}_{g})\int^1_0\frac{\partial\mathcal{M}_{\nu}(\theta)}{\partial \mathcal{T}_{\nu}}d\theta\cr
&\equiv I_1+I_2+I_3,
\end{split}
\end{eqnarray}
where
\begin{eqnarray*}
\frac{\partial\mathcal{M}_{\nu}(\theta)}{\partial X}=\frac{\partial\mathcal{M}_{\nu}}{\partial X}
(\rho_{\theta}, U_{\theta}, \mathcal{T}_{\theta})
\end{eqnarray*}
for $(\rho_{\theta}, U_{\theta}, \mathcal{T}_{\theta})=(1-\theta)\big(\rho_{f}, U_{f}, \mathcal{T}_{f}\big)
+\theta\big(\rho_{g}, U_{g}, \mathcal{T}_{g}\big)$.
Since $(\rho_{\theta}, U_{\theta}, \mathcal{T}_{\theta})$ is a linear combination of macroscopic fields of $f$ and $g$,
all lemmas in the previous sections hold the same. Therefore, instead of restating the corresponding lemmas, we refer to them whenever such estimates are needed for $(\rho_{\theta}, U_{\theta}, \mathcal{T}_{\theta})$.\newline\newline
%
%
%
%
\noindent(a) Estimate for $I_1$: Since we have
\begin{eqnarray*}
\frac{\partial\mathcal{M}_{\nu}(\theta)}{\partial\rho}=\frac{1}{\rho_{\theta}}\mathcal{M}_{\nu}(\theta),
\end{eqnarray*}
it follows directly from $\rho_{\theta}\geq a_{\ell}$ and Lemma \ref{Max decomposition} that
\begin{eqnarray}
\left|\frac{\partial\mathcal{M}_{\nu}(\theta)}{\partial\rho}\right|\leq  C_{\ell,u}e^{-C_{\ell,u}|v|^2}.
\end{eqnarray}
\noindent(b) Estimate for $I_2$: An explicit computation gives
\begin{eqnarray*}
\frac{\partial\mathcal{M}_{\nu}(\theta)}{\partial U}
=-\frac{1}{2}\Big\{(v-U_{\theta})^{\top}\mathcal{T}^{-1}_{\theta}+\mathcal{T}^{-1}_{\theta}(v-U_{\theta})\Big\}
\mathcal{M}_{\nu}(\theta).
\end{eqnarray*}
Let $X=v-U_{\theta}$ and observe
\begin{align*}
|X^{\top}\mathcal{T}_{\theta}^{-1}|&=\sup_{|Y|=1}X^{\top}\{\mathcal{T}_{\theta}\}^{-1}Y\cr
&=\frac{1}{2}\sup_{|Y|=1}\Big\{(X+Y)^{\top}\{\mathcal{T}_{\theta}\}^{-1}(X+Y)-X^{\top}\{\mathcal{T}_{\theta}\}^{-1}X-Y^{\top}\{\mathcal{T}_{\theta}\}^{-1}Y\Big\}\cr
&\leq C\left(\frac{|X+Y|^2+|X|^2+1}{T_{\theta}}\right)\cr
&\leq C\left(\frac{1+|v-U_{\theta}|^2}{T_{\theta}}\right),
\end{align*}
which is, by Lemma \ref{UT lemma},  bounded by  $C_{\ell,u}(1+|v|^2)$. Similarly, we can derive
\begin{align*}
|\{\mathcal{T}_{\theta}\}^{-1}(v-U_{\theta })|\leq C_{\ell,u}(1+|v|^2).
\end{align*}
With these computations and Lemma \ref{UT lemma} and Lemma \ref{Max decomposition}, we have
\begin{eqnarray*}
\Big|\frac{\partial\mathcal{M}_{\nu}(\theta)}{\partial U}\Big|
\leq C_{\ell,u}\mathcal{M}_{\nu}(\theta)(1+|v|^2)\leq  C_{\ell,u}e^{-C_{\ell,u}|v|^2}.
\end{eqnarray*}
\noindent(c) Estimate for $I_3$: We first observe
\begin{eqnarray*}
\frac{\partial \mathcal{M}_{\nu}(\theta)}{\partial \mathcal{T}_{ij}}
=\frac{1}{2}\left[-\frac{1}{\det\mathcal{T}_{\theta}}\frac{\partial\det\mathcal{T}_{\theta}}{\partial\mathcal{T}_{\theta ij}}
+(v-U_{\theta})^{\top}\mathcal{T}_{\theta}^{-1}\left(\frac{\partial\mathcal{T}_{\theta}}{\partial{\mathcal{T}_{ij}}}\right)\mathcal{T}_{\theta}^{-1}(v-U_{\theta})\right]
\mathcal{M}_{\nu}(\theta).
\end{eqnarray*}
Since each entry of $\frac{\partial\mathcal{T}_{\theta}}{\partial{\mathcal{T}_{\theta ij}}}$ is either 1 or 0, we have
\begin{eqnarray}\label{use3}\begin{split}
\qquad\left|(v-U_{\theta})^{\top}\mathcal{T}^{-1}_{\theta}\left(\frac{\partial\mathcal{T}_{\theta}}{\partial{\mathcal{T}_{\theta ij}}}\right)\mathcal{T}^{-1}_{\theta}
(v-U_{\theta})\right|
&\leq \left|(v-U_{\theta})^{\top}\mathcal{T}^{-1}_{\theta}\right|
\left|\mathcal{T}^{-1}_{\theta}(v-U_{\theta})\right|\cr
\qquad&\leq C_{\ell,u}(1+|v|^2).
\end{split}
\end{eqnarray}
Since $\det\mathcal{T}_{\theta}$ is a homogeneous polynomial of entries of $\mathcal{T}_{\theta}$:
\[
\sum_{i,j,k,\ell,m,n} C_{ijk\ell mn}\mathcal{T}_{\theta ij}\mathcal{T}_{\theta k\ell}\mathcal{T}_{\theta mn},
\]
for some constants $C_{ijk\ell mn}$, $\frac{\partial\det\mathcal{T}_{\theta}}{\partial\mathcal{T}_{\theta ij}}$ is written in the following form.
\[
\sum_{i,j,m,n} C_{ijmn}\mathcal{T}_{\theta ij}\mathcal{T}_{\theta mn}
\]
for some constants $C_{ijmn}$. Therefore, in view of Lemma \ref{equiv} and Lemma \ref{UT lemma}, we have
\begin{align*}
\left|\frac{\partial\det\mathcal{T}_{\theta}}{\partial\mathcal{T}_{\theta ij}}\right|\leq CT^2_{\theta}\leq C_{\ell,u}.
\end{align*}
Hence, Lemma \ref{Max decomposition} yields
\[
\Big|\frac{\partial \mathcal{M}_{\nu}(\theta)}{\partial \mathcal{T}_{ij}}\Big|\leq C_{\ell,u}(1+|v|^2)\mathcal{M}_{\nu}(\theta)\leq C_{\ell,u}e^{-C_{\ell,u}|v|^2}.
\]
Plugging all these estimates into (\ref{turnbackto}) gives
\begin{align}\label{obtain}
\begin{split}
&|\mathcal{M}_{\nu}(f)-\mathcal{M}_{\nu}(g)|\cr
&\qquad\leq C_{\ell,u}\Big\{|\rho_{f}-\rho_{g}|+|U_{f}-U_{g}|
+|\mathcal{T}_{f}-\mathcal{T}_{g}|\Big\} e^{-C_{\ell,u}|v|^2}.
\end{split}
\end{align}
It remains to estimate the macroscopic fields. The first term is estimated straightforwardly:
\begin{eqnarray*}
|\rho_{f}-\rho_{g}|=\int_{\mathbb{R}^3} |f-g|dv\leq C\sup_x\|f-g\|_{L^1_2}.
\end{eqnarray*}
We divide the second term into two parts and estimate separately as
\begin{eqnarray*}
|U_{f}-U_{g}|
&\leq&\frac{1}{\rho_f}|\rho_fU_f-\rho_gU_g|+\frac{1}{\rho_f}|\rho_f-\rho_g||U_g|\cr
&\leq&\frac{1}{\rho_{f}}\int_{ \mathbb{R}^3}|f-g||v|dv+\frac{|U_g|}{\rho_f}\int_{\mathbb{R}^3}|f-g|dv\cr
&\leq&C_{\ell,u}\sup_x\|f-g\|_{L^{1}_2}.
\end{eqnarray*}
The last term is decomposed similarly:
\begin{eqnarray*}
|\mathcal{T}_{f}-\mathcal{T}_{g}|
\leq\frac{1}{\rho_f}|\rho_f\mathcal{T}_{f}-\rho_g\mathcal{T}_{g}|
+\frac{1}{\rho_f}|\rho_f-\rho_g||\mathcal{T}_{g}|=J_1+J_2,
\end{eqnarray*}
where $J_1$ and $J_2$ are computed as
\begin{align*}
J_1
&\leq\frac{1}{a_{\ell}}\int_{\mathbb{R}^3}|f-g|\left|3^{-1}(1-\nu)|v-U|^2Id+\nu(v-U)\otimes (v-U)\right|dv\cr
&\leq \frac{1}{a_{\ell}}\int_{\mathbb{R}^3}|f-g|(1+|v|^2)dv\cr
&\leq C_{\ell,u}\sup_x\|f-g\|_{L^{1}_2},
\end{align*}
and
\begin{eqnarray*}
J_2\leq C_{\ell,u}\int_{\mathbb{R}^3}|f-g|dv
\leq C_{\ell,u}\sup_x\|f-g\|_{L^1_2}.
\end{eqnarray*}
We now substitute these estimates into (\ref{obtain}) to obtain
\begin{eqnarray*}
|\mathcal{M}_{\nu}(f)-\mathcal{M}_{\nu}(g)|\leq C_{\ell,u}\sup_x\|f-g\|_{L^{1}_2}e^{-C_{\ell,u}|v|^2}.
\end{eqnarray*}
\end{proof}

\begin{proposition}\label{fixed point 2}
Suppose $f,g\in\Omega$. Then, under the assumption of Theorem 2.2, $\Phi$ satisfies
\[
\sup_{x\in[0,1]}\|\Phi(f)-\Phi(g)\|_{L^1_2}\leq \alpha\sup_{x\in[0,1]}\|f-g\|_{L^1_2}
\]
for some constant $\alpha<1$ depending on the quantities in (\ref{quantities}), $\gamma_{\ell}$, $\nu$ and $\kappa$.
\end{proposition}
\begin{proof}
We first consider $\Phi^+(f)$. We write
\[
\Phi^+(f)=I(f)+II(f,f,f),
\]
where $I(f)$ and $II(f,g,h)$ are defined by
\begin{eqnarray*}
I(f)=e^{-\frac{1}{\tau|v_1|}\int^x_0\rho_f(y)dy}f_L(v),
\end{eqnarray*}
and
\begin{eqnarray*}
II(f,g,h)=\frac{1}{\tau|v_1|}\int^x_0e^{-\frac{1}{\tau|v_1|}\int^x_y\rho_f(z)dz}\rho_g(y)\mathcal{M}_{\nu}(h)dy.
\end{eqnarray*}
{\bf(i)} The estimate for $I(f)-I(g)$: Consider
\[
I(f)-I(g)=\Big\{e^{-\frac{1}{\tau|v_1|}\int^x_0\rho_f(y)dy}
-e^{-\frac{1}{\tau|v_1|}\int^x_0\rho_g(y)dy}\Big\}f_L(v).
\]
By the mean value theorem, there exists $0<\theta<1$ such that
\begin{eqnarray*}
&&e^{-\frac{1}{\tau|v_1|}\int^x_0\rho_f(y)dy}
-e^{-\frac{1}{\tau|v_1|}\int^x_0\rho_g(y)dy}\cr
&&\qquad=-e^{-\frac{1}{\tau|v_1|}\int^x_0(1-\theta)\rho_f(y)+\theta\rho_g(y)dy}
\left\{\frac{1}{\tau|v_1|}\int^x_0\rho_f(y)-\rho_g(y)dy\right\}.
\end{eqnarray*}
Then, since
\begin{align*}
|\rho_f(y)-\rho_g(y)|\leq \sup_{x\in[0,1]}\|f-g\|_{L^1_2},
\end{align*}
and $\rho_f,\,\rho_g\geq a_{\ell}$, we have
\begin{align*}
&\left|\,e^{-\frac{1}{\tau|v_1|}\int^x_0\rho_f(y)dy}
-e^{-\frac{1}{\tau|v_1|}\int^x_0\rho_g(y)dy}\right|\cr
&\qquad\leq e^{-\frac{1}{\tau|v_1|}\int^x_0(1-\theta)\rho_f(y)+\theta\rho_g(y)dy}
\frac{1}{\tau|v_1|}\int^x_0|\rho_f(y)-\rho_g(y)|dy\cr
&\qquad\leq e^{-\frac{1}{\tau|v_1|}\int^x_0(1-\theta)\rho_f+\theta \rho_gdy}
\left\{\frac{x}{\tau|v_1|}\sup_{x\in[0,1]}\|f-g\|_{L^1_2}\right\}\cr
&\qquad\leq\frac{1}{\tau|v_1|}e^{-\frac{1}{\tau|v_1|}\int^x_0(1-\theta)a_{\ell}+\theta a_{\ell}dy}
\sup_{x\in[0,1]}\|f-g\|_{L^1_2}\cr
&\qquad=\frac{1}{\tau|v_1|}e^{-\frac{a_{\ell}x}{\tau|v_1|}}
\sup_{x}\|f-g\|_{L^1_2}.
\end{align*}
Using this, we integrate
\begin{align*}
&\int_{\mathbb{R}^3}|I(f)-I(g)|(1+|v|^2)dv\cr
&\quad\leq \int_{v_1>0}\Big|e^{-\frac{1}{\tau|v_1|}\int^x_0\rho_f(y)dy}
-e^{-\frac{1}{\tau|v_1|}\int^x_0\rho_g(y)dy}\Big|f_L(v)(1+|v|^2)dv\cr
&\quad\leq\left\{\int_{v_1>0}\frac{1}{\tau|v_1|}e^{-\frac{a_{\ell}x}{\tau|v_1|}}
f_L(v)(1+|v|^2)dv\right\}\sup_{x}\|f-g\|_{L^1_2}\cr
&\quad\leq\frac{1}{\tau}\left(a_s+c_s\right)\sup_{x}\|f-g\|_{L^1_2}.
\end{align*}
Taking supreme in $x$, we have
\[
\sup_{x}\|I(f)-I(g)\|_{L^1_2}\leq \frac{1}{\tau}(a_s+c_s)\sup_{x}\|f-g\|_{L^1_2}.
\]
{\bf(ii)} The estimate for $II(f)-II(g)$: We divide it into three parts as
\begin{eqnarray*}
&&II(f,f,f)-II(g,g,g)\cr
&&\quad=\big\{II(f,f,f)-II(g,f,f)\big\}+\big\{II(g,f,f)-II(g,g,f)\big\}\cr
&&\quad+\big\{II(g,g,f)-II(g,g,g)\big\}\cr
&&\quad=II_1+II_2+II_3.
\end{eqnarray*}
By a similar manner as for $I(f)$, we first compute
\begin{eqnarray*}
&&\Big|\,e^{-\frac{1}{\tau|v_1|}\int^x_y\rho_f(z)dz}-e^{-\frac{1}{\tau|v_1|}\int^x_y\rho_g(z)dz}\Big|\cr
&&\hspace{1.5cm}\leq e^{-\frac{1}{\tau|v_1|}\int^x_y(1-\theta)\rho_f(z)+\theta\rho_g(z)dz}
\frac{1}{\tau|v_1|}\int^x_y|\rho_g(z)-\rho_g(z)|dz\cr
&&\hspace{1.5cm}\leq e^{-\frac{1}{\tau|v_1|}\int^x_y(1-\theta)a_{\ell}+\theta a_{\ell}dz}
\left\{\int^x_y\frac{1}{\tau|v_1|}dz\right\}\|\rho_f-\rho_g\|_{L^{\infty}_x}\cr
&&\hspace{1.5cm}\leq\frac{x-y}{\tau|v_1|}e^{-\frac{a_{\ell }(x-y)}{ \tau|v_1|}}
\sup_{x}\|f-g\|_{L^1_2}\cr
&&\hspace{1.5cm}\leq\frac{C}{a_{\ell}}e^{-\frac{a_{\ell }(x-y)}{ 2\tau|v_1|}}
\sup_{x}\|f-g\|_{L^1_2},
\end{eqnarray*}
where we used the uniform boundedness of $xe^{-x}$ $(x>0)$. With this, (\ref{I}) and Lemma \ref{Max decomposition}, we bound $\int_{\mathbb{R}^3}|II_1|(1+|v|^2)dv$ by
\begin{eqnarray*}
&&\int_{\mathbb{R}^3}|II_1|(1+|v|^2)dv\cr
&&\qquad\leq\int_{\mathbb{R}^3}\int^x_0\frac{1}{\tau|v_1|}\Big|~e^{-\frac{1}{\tau|v_1|}\int^x_y\rho_f(z)dz}
-e^{-\frac{1}{\tau|v_1|}\int^x_y\rho_g(z)dz}\Big|~\rho_f(y)\mathcal{M}_{\nu}(f)(1+|v|^2)dydv\cr
&&\qquad=  \frac{a_u}{a_{\ell}}
\left\{\int_{v_1>0}\int^x_0\frac{1}{\tau|v_1|}e^{-\frac{a_{\ell}(x-y)}{2\tau |v_1|}}
\mathcal{M}_1(v_1)dydv_1\right\}\sup_{x}\|f-g\|_{L^1_2}\cr
&&\qquad\leq C_{\ell,u}\left(\frac{\ln \tau+1}{\tau}\right)\sup_{x}\|f-g\|_{L^1_2}.
\end{eqnarray*}
We can treat $II_2$ similarly:
\begin{eqnarray*}
&&\hspace{-1cm}\int_{\mathbb{R}^3}|II_2|(1+|v|^2)dv\cr
&&=\int_{\mathbb{R}^3}\frac{1}{\tau|v_1|}\int^x_0e^{-\frac{1}{\tau|v_1|}\int^x_y\rho_g(z)dz}| \rho_f(y)-\rho_g(y)|\mathcal{M}_{\nu}(f)(1+|v|^2)dydv\cr
&&\leq C_{\ell,u}\int^x_0\left\{\int_{\mathbb{R}^3}\frac{1}{\tau|v_1|}e^{-\frac{1}{\tau|v_1|}\int^x_y\rho_g(z)dz}\mathcal{M}(f)(1+|v|^2)dv\right\}
|\rho_f(y)-\rho_g(y)|dy\cr
&&\leq C_{\ell,u}\left\{\int^x_0\int_{v_1>0}\frac{1}{\tau|v_1|}e^{-\frac{1}{\tau|v_1|}\int^x_y\rho_g(z)dz}\mathcal{M}_1(v_1)dv_1dy\right\}
\sup_{x}\|f-g\|_{L^1_2}\cr
&&\leq C_{\ell,u}\left\{\int^x_0\int_{v_1>0}\frac{1}{\tau|v_1|}e^{-\frac{a_{\ell}(x-y)}{\tau|v_1|}}\mathcal{M}_1(v_1)dv_1dy\right\}
\sup_{x}\|f-g\|_{L^1_2}\cr
&&\leq C_{\ell,u}\left(\frac{\ln \tau+1}{\tau}\right)\sup_{x}\|f-g\|_{L^1_2}.
\end{eqnarray*}
For the estimate of $II_3$, we use Proposition \ref{Lipshitz} as
\begin{align*}
&\hspace{-0.5cm}\int_{\mathbb{R}^3}|II_3|(1+|v|^2)dv\cr
&=\int_{\mathbb{R}^3}\frac{1}{\tau|v_1|}\int^x_0e^{-\frac{1}{\tau|v_1|}\int^x_y\rho(z)dz}\rho_g(y)| \mathcal{M}_{\nu}(f)-\mathcal{M}_{\nu}(g)|(1+|v|^2)dydv\cr
&\leq a_uC_{\ell,u}\left\{\int_{v_1>0}\frac{1}{\tau|v_1|}\int^x_0e^{-\frac{1}{\tau|v_1|}\int^x_y\rho(z)dz}e^{-C_{\ell,u}|v|^2}(1+|v|^2)dvdy\right\}
\sup_{x}\|f-g\|_{L^1_2}\cr
&\leq C_{\ell,u} \left\{\int^x_0\int_{v_1>0}\frac{1}{\tau|v_1|}e^{-\frac{a_{\ell}(x-y)}{\tau|v_1|}}
e^{-C_{\ell,u}|v|^2}dv_1dy\right\}
\sup_{x}\|f-g\|_{L^1_2}.
\end{align*}
Therefore, in view of (\ref{I}), we have
\begin{align*}
II_3\leq C_{\ell,u}\left(\frac{\ln \tau+1}{\tau}\right)\sup_{x}\|f-g\|_{L^1_2}.
\end{align*}
We now gather all these estimates to obtain
\begin{eqnarray*}
\sup_{x}\|\Phi^+(f)-\Phi^+(g)\|_{L^1_2}
\leq C_{\ell,u}\left\{\frac{1}{\tau}(a_s+c_s) +C_{\tau}\right\}\sup_{x}\|f-g\|_{L^1_2},
\end{eqnarray*}
where
\[
C_{\tau}=\frac{\ln \tau+1}{\tau}.
\]
In a similar fashion, we can derive the corresponding estimate for $\Phi^-(f)$:
\begin{eqnarray*}
\sup_{x}\|\Phi^-(f)-\Phi^-(g)\|_{L^1_2}
\leq C_{\ell,u}\left\{\frac{1}{\tau}(a_s+c_s) +C_{\tau}\right\}\sup_{x}\|f-g\|_{L^1_2}.
\end{eqnarray*}
Therefore, we conclude that
\begin{align*}
\sup_{x}\|\Phi(f)-\Phi(g)\|_{L^1_2}
\leq C_{\ell,u}\left\{\frac{1}{\tau}\big(a_s+c_s\big) +C_{\tau}\right\}\sup_{x}\|f-g\|_{L^1_2}.
\end{align*}
This gives the desired result for sufficiently large $\tau>0$.
\end{proof}
\begin{section}{Acknowledgement}
The work of S.-B. Yun was supported by Basic Science Research Program through
the National Research Foundation of Korea (NRF) funded by the Ministry of Science, ICT $\&$
Future Planning (NRF-2014R1A1A1006432)
\end{section}
\bibliographystyle{amsplain}

\end{document}